\documentclass[leqno,11pt,draft]{amsart}
\usepackage{amssymb}
\usepackage{amsmath}
\usepackage{enumerate}
\usepackage{amsfonts}

\newtheorem{teo}[equation]{Theorem}
\newtheorem{coro}[equation]{Corollary}
\newtheorem{lema}[equation]{Lemma}

\newtheorem{remark}[equation]{Remark}

\numberwithin{equation}{section}

\newcommand{\dtt}{\frac{dt}{\sqrt{t}}}
\newcommand{\Q}{\mathcal{Q}}
\renewcommand{\L}{\mathcal{L}}
\newcommand{\N}{\mathbb{N}}
\newcommand{\R}{\mathbb{R}}
\renewcommand{\(}{\left(}
\renewcommand{\)}{\right)}
\newcommand{\8}{\infty}
\renewcommand{\dj}{\frac{\partial}{\partial x_j}}
\newcommand{\e}{\varepsilon}

\begin{document}

\title[Riesz transforms and Hardy spaces ]
{On Riesz transforms characterization
 of $H^1$ spaces\\
associated with some Schr\"odinger operators}


\begin{abstract} Let $\mathcal Lf(x)=-\Delta f (x)+V(x)f(x)$,
$V\geq 0$, $V\in L^1_{loc}(\R^d)$,
 be a non-negative self-adjoint Schr\"odinger operator on $\mathbb R^d$. We say that
 an $L^1$-function $f$ belongs to the Hardy space $H^1_{\mathcal L}$
 if the maximal function
 $$\mathcal M_{\mathcal L} f(x)=\sup_{t>0}|e^{-t\mathcal L} f(x)|$$ belongs to $L^1(\mathbb R^d)$.
 We prove that under certain   assumptions on $V$  the space $H^1_{\mathcal L}$ is
  also characterized by the Riesz transforms
 $R_j=\frac{\partial}{\partial x_j}\mathcal L^{-1\slash 2}$, $j=1,...,d$,
  associated with $\mathcal L$.
   As an example of such a potential $V$ one can take any $V\geq 0$, $V\in L^1_{loc}$, in one dimension.
\end{abstract}

\author[Jacek Dziuba\'nski ]{ Jacek Dziuba\'nski }
\address{Instytut Matematyczny\\
Uniwersytet Wroc\l awski\\
50-384 Wroc\l aw \\
Pl. Grunwal\-dzki 2/4\\
Poland} \email{jdziuban@math.uni.wroc.pl,
preisner@math.uni.wroc.pl}

\author[Marcin Preisner ]{ Marcin Preisner}

\thanks{
 The research partially
  supported by Polish Government funds for science - grant N N201 397137, MNiSW}

\maketitle

{\it In memory of Andrzej Hulanicki}.

\

\section{Introduction.}

On $\mathbb R^d$ we consider a Schr\"odinger operator $\mathcal
L=-\Delta + V(x)$, where $V(x)$ is a locally integrable
nonnegative potential, $V\not\equiv 0$. It is well known that
$-\mathcal L$ generates the semigroup $\{T_t\}_{t>0}$ of linear
contractions on $L^p(\mathbb R^d)$, $1\leq p<\infty$. The
Feynman-Kac formula implies that the integral kernels $T_t(x,y)$
of this semigroup satisfy
\begin{equation}\label{kac}
 0\leq T_t(x,y)\leq P_t(x-y)= \left(4\pi t\right)^{-d/2} \exp(-|x-y|^2\slash
 4t).
\end{equation}
 We say that an $L^1$-function $f$ belongs to the Hardy space
 $H^1_{\mathcal L}$ if the maximal function
 $$\mathcal M_{\mathcal L} f(x)=\sup_{t>0}|T_tf(x)|$$ belongs to $L^1(\mathbb R^d)$.
 We set
 $$ \| f\|_{H^1_{\mathcal L}}=\| \mathcal M_{\mathcal
 L}f\|_{L^1(\mathbb R^d)}.$$

Let $\mathcal{Q} = \{ Q_j\}_{j=1}^{\infty}$ be a family  that
consists of closed cubes with disjoint interiors such that
$\mathbb R^d$ is the closure of  ${\bigcup_{j=1}^\infty Q_j}$. We
shall always assume that there exist constants $C, \beta>0$ such
that if $Q_i^{****} \cap Q_j^{****} \neq \emptyset$ then $d(Q_i)
\leq C d(Q_j)$, where $d(Q)$ denotes the diameter of $Q$, and
$Q^*$ is the cube with the same center as $Q$ such that $d(Q^*) =
(1+\beta) d(Q)$. Clearly, there is a constant $C>0$ such that
\begin{equation}\label{cov}
\sum_{j=1}^\infty \mathbf 1_{Q_j^{****}}(x)\leq C.
\end{equation}

In order to state results from \cite{DZ5} we recall the notion of
the local Hardy space associated with the collection $\mathcal{Q}$.
We say that a function $a$ is an $H^1_\mathcal{Q}$-atom if there
exists $Q\in \Q$ such that either $a = |Q|^{-1} \mathbf 1 _Q$ or
$a$ is the classical atom with support contained in $Q^{*}$ (that
is, there is a cube $Q'\subset Q^*$
 such that $\text{supp}\, a\subset Q'$, $\int a=0$,
 $|a|\leq |Q'|^{-1}$).

 The atomic space $H^1_{\Q}$ is defined by
 \begin{equation}\label{eq14} H^1_{\Q}
 =\big\{ f: f=\sum_j \lambda_j a_j, \ \sum_j |\lambda_j|<\infty \big\},
 \end{equation}
 where $\lambda_j\in\mathbb C$, $a_j$ are $H^1_{\Q}$-atoms. We set
 $$ \| f\|_{H^1_\Q}=\inf\big\{\sum_{j}|\lambda_j|\big\},$$
 where the infimum is taken over all representations of $f$ as in
 (\ref{eq14}).

Following \cite{DZ5} we will also impose  two additional
assumptions on the potential $V$ and the collection $\Q$ of cubes,
mainly:
\begin{equation}\tag*{(D)}\label{D}
(\exists C, \e >0) \qquad \sup_{y\in Q^*} \int T_{2^n
d(Q)^2}(x,y)dx \leq Cn^{-1-\e} \quad \text{for} \  Q\in \Q, n \in
\N;
\end{equation}

\begin{equation}\tag*{(K)}\label{K}
(\exists C, \delta >0) \ \ \   \int_0^{2t} (\mathbf {1}
_{Q^{***}}V)*P_s(x) ds \leq C\left(\frac{t}{d(Q)^2}\right)^\delta
\quad \text{for} \ x\in \R^d,   \  Q\in \Q, t\leq d(Q)^2.
\end{equation}

Thorem 2.2  of \cite{DZ5} states that if we assume \ref{D} and
\ref{K} then we have the following atomic characterization of the
Hardy space $H^1_\L$:
\begin{equation}\label{atoms}
 f\in H^1_{\L} \iff f\in H^1_{\Q}. \quad \text {Moreover,} \quad C^{-1}\|f\|_{H^1_\Q} \leq  \|f\|_{H^1_\L} \leq C\|f\|_{H^1_\Q}.
\end{equation}

 For $j=1, \dots , d$, let
 $$R_j f (x)=\lim_{\varepsilon \to 0}\int_{\varepsilon}^{1\slash
 \varepsilon }\dj T_t f(x)\frac{dt}{\sqrt{t}}$$
 be the Riesz transform $\dj \mathcal L^{-1\slash 2}$
 associated with $\mathcal L$, where the limit is understood in
 the sense of distributions (see Section 2). The main result of this paper is to
 prove that, under these conditions, the operators $R_j$ characterize the space
 $H^1_{\mathcal L}$, that is, the following theorem holds.

 \begin{teo}\label{main}
Assume that a potential $V\geq 0$ and a collection of cubes $\Q$
are such that  \ref{D} and \ref{K} hold. Then there exists a
constant $C>0$ such that
 \begin{equation}\label{eq18}
 C^{-1} \| f\|_{H^1_{\mathcal L}}\leq  \|f\|_{L^1(\mathbb
 R^d)}+ \sum_{j=1}^d \| R_j f\|_{L^1(\mathbb R^d)} \leq C\|
 f\|_{H^1_{\mathcal L}}.
 \end{equation}
 \end{teo}

 \

 \begin{remark} {\rm For $\ell >0$ denote by $\mathcal Q_{\ell}(\mathbb R^n)$ a
 partition of $\mathbb R^n$ into cubes whose diameters have length
 $\ell$. Assume that for a locally integrable nonnegative potential
  $ V_1$ on $\mathbb R^d$ and a collection $\mathcal Q$ of cubes
  the conditions (D) and (K) hold. Consider
  the potential $V(x_1,x_2)=V_1(x_1)$, $x_1\in \mathbb R^d$,
  $x_2\in\mathbb R^n$, and the family $\mathcal {\widetilde Q}=\{Q_1\times Q_2:\,
  Q_1\in\mathcal Q, \ Q_2\in\mathcal Q_{d(Q_1)}(\mathbb R^n)\}$ of
  cubes in $\mathbb R^{d+n}$. It is easily seen that the pair
  $(V,\mathcal {\widetilde Q})$ fulfils (D) and (K).}
 \end{remark}
  \begin{remark} {\rm One can check that Theorem 2.2 of \cite{DZ5} (see (\ref{atoms}))
  and Theorem \ref{main} together with their proofs remain true if
  we replace cubes by rectangles in the definition of atoms and in the
  conditions (D) and
  (K), provided the rectangles have side-lengths comparable to
  their diameters. As a corollary of this observation  we obtain that if
   $V(x_1,x_2)=V_1(x_1)+V_2(x_2)$,
  $x_1\in\mathbb R^d$, $x_2\in\mathbb R^n$, where $V_1$ and $V_2$
  satisfy conditions (D) and (K) for certain
  collections $\mathcal Q_1$ and $\mathcal Q_2$ of cubes on $\mathbb R^d$
  and $\mathbb R^n$ respectively,
  then the Hardy space $H^1_{\mathcal L}$ associated with the operator $\mathcal
  L=-\Delta +V(x_1,x_2)$ in $\mathbb R^{d+n}$ admits the atomic and the Riesz transforms
  characterizations. Indeed, for any
  $Q_j\in\mathcal Q_1$ and $Q_k\in\mathcal Q_2$ we divide the
  rectangle
  $Q_j\times Q_k$ into rectangles $Q_{j,k}^{s}$, $s=1,2,..., s_{j,k}$, with side-lengths comparable
  to $\min (d(Q_j^1), d(Q_k^2))$. It is not difficult to verify that (D) and (K) hold for
  $V(x_1,x_2)$ and   the collection $Q_{j,k}^s$.}
 \end{remark}

 \

 {\bf Examples.} We finish the section by recalling some examples of nonnegative
 potentials $V$ considered in  \cite{CZ} and  \cite{DZ5}.
 such that the semigroups generated by $\Delta -V$ satisfy  (D)
 and
 (K) for relevant collections $\mathcal Q$ of cubes.

 $\bullet$
 The Hardy space $H^1_{\mathcal L}$  associated with one-dimensional Schr\"odinger
 operator $-\mathcal L$ was studied in Czaja-Zienkiewicz \cite{CZ}. It was proved
 there that for any nonnegative $V\in L^1_{loc}(\mathbb R)$   the collection $\mathcal Q$
 of maximal dyadic intervals $Q$
 of $\R$ that are defined by the stopping time condition
 \begin{equation}\label{eq19} |Q|\int_{16Q }V(y)\, dy\leq 1,
 \end{equation}
 fulfils  (D)  for certain small  $\beta >0$ (see \cite[Lemma
 2.2]{CZ}). The authors also remarked that (K) is satisfied.
 Indeed,
 \begin{equation*}\label{eq88}\begin{split} \int_0^{2t} (\mathbf 1_{Q^{***}} V)*P_s(x)\, ds\leq
 \int_0^{2t} \| \mathbf 1_{Q^{***}}V\|_{L^1}\| P_s\|_{L^{\infty}}\, ds\leq
  \int_0^{2t}|Q|^{-1}\frac{ds}{\sqrt{4\pi s}}\leq
  C\frac{t^{1\slash 2}}{|Q|},
  \end{split}\end{equation*}
 where in the second inequality we have used (\ref{eq19}).

 $\bullet$ $V(x)=\gamma |x|^{-2}$, $d\geq 3$, $\gamma >0$. Then
 for $\mathcal Q$ being the Whitney decomposition of $\mathbb
 R^d\setminus \{0\}$ that consists of dyadic cubes the conditions
 (D) and (K) hold (see Theorem 2.8 of \cite{DZ5}).

 $\bullet$ $d\geq 3$, $V$ satisfies the reverse H\"older inequality
 with exponent $q > d\slash 2$, that is,
 $$\left(\frac{1}{|B|}\int_B V(y)^q\, dy\right)^{1\slash q}\leq C\frac{1}{|B|}
 \int_B V(y)\, dy \ \  \ \text{for every ball } B. $$
 Define the  family $\mathcal Q$
  by: $Q\in\mathcal Q$ if and only if $Q$ is the maximal
 dyadic cube for which $d(Q)^2|Q|^{-1}\int_Q V(y)\, dy \leq 1$.
 Then the
 conditions (D) and (K) are true (see \cite[Section 8]{DZ5}).

\

 Let us finally mention that the Riesz transforms characterization of
 the  Hardy spaces associated with
 Schr\"odinger operators with potentials satisfying the reverse
 H\"older inequality was proved in \cite{DZ1}.
\section{Auxiliary estimates}

 \begin{lema}\label{lem2.1}
 For every $\alpha >0$ there exists a constant $C>0$ (independent of $V$) such that for $j=1, \dots, d$ and
 $y\in\mathbb R^d$ we have
 \begin{equation}\label{EF1}
 \int_{\mathbb R} \left|\dj T_t(x,y)\right |^2\exp(\alpha |x-y|\slash
 \sqrt{t})\, dx \leq C t^{-d\slash 2-1},
 \end{equation}
 \begin{equation}\label{EF11}
\int_{\mathbb R} \left|\dj T_t(x,y)\right |\exp(\alpha |x-y|\slash
 \sqrt{t})\, dx\leq C t^{-1\slash 2}.
 \end{equation}

 \end{lema}
The lemma is known. For reader's  convenience we give a sketch of
a proof in Section 4.

 For $\varepsilon >0$, $j=1, \dots , d$, we define the operator $$R_j^\varepsilon f(x)=\int
 R_j^\varepsilon (x,y) f(y)\, dy,$$
 where $R_j^\varepsilon (x,y)=\int_{\varepsilon}^{1\slash
 \varepsilon} \dj
 T_t(x,y)\frac{dt}{\sqrt{t}}$.
 It is not difficult to see that for $f\in L^1(\mathbb R^d)$ the
 limits $\lim_{\varepsilon \to 0} R_j^\varepsilon f(x)$ exist in the sense of
 distributions and define tempered distributions which will be denoted by
   $R_jf$. Moreover, for $\varphi \in \mathcal
 S(\mathbb R^d)$ we have
  \begin{equation}\label{estimate_R}
  |\langle R_j f, \varphi \rangle | \leq C \| f\|_{L^1(\mathbb R^d)}
 \left( \| \varphi\|_{L^2(\mathbb R^d)}+\Big \|\dj \varphi\Big\|_{L^\infty
 (\mathbb R^d)}\right). \end{equation}
 To see this we write
 \begin{equation*}
 R_j^{\varepsilon\,  *} \varphi (y)=\int_1^{1\slash \varepsilon}
 \int_{\mathbb R^d}
 \dj T_t(x,y) \varphi (x)\, dx \,
 \frac{dt}{\sqrt{t}}-\int_{\varepsilon}^1 \int_{\mathbb R^d} T_t(x,y)  \dj \varphi
 (x)\, dx\, \frac{dt}{\sqrt{t}}.
 \end{equation*}
 Since $$\int_1^\infty
 \left[ \int_{\R^d}\left| \dj T_t(x,y)\right|^2 \,
 dx\right]
 ^{\frac{1}{2}}
 \frac{dt}{\sqrt{t} } \leq C\int_1^\infty t^{-1-\frac{d}{4}} \, dt\leq
 C$$
 and
$$ \int_{\R^d}\int_0^1T_t(x,y)\, \frac{dt}{\sqrt{t}} \, dx \leq
 2$$
 (see Lemma \ref{lem2.1}),  we conclude  that  $
 R_j^{\varepsilon \, *} \varphi (y)$ converges uniformly, as $\varepsilon \to
 0$,  to a
 bounded function which will be denoted by $R_j^* \varphi (y)$,
 and
 \begin{equation*}\label{estimate-star}
 |R_j^* \varphi (y)| \leq C \left(\| \varphi\|_{L^2(\R^d)}+\Big\| \dj \varphi
 \Big\|_{L^\infty(\R^d)}\right).
 \end{equation*}

 For fixed $Q\in\mathcal \Q$ and $0 <\varepsilon < 1 $, let
 \begin{equation*}
 R_{j,Q,0}^\varepsilon
 (x,y)=\begin{cases}
 \int_{\varepsilon}^{d(Q)^2}\dj T_t(x,y)\frac{dt}{\sqrt{t}}
 \ \ \ & \text{if } \ \varepsilon <d(Q)^2<1\slash \varepsilon ;\\
 \int_\varepsilon^{1\slash \varepsilon}\dj T_t(x,y)\frac{dt}{\sqrt{t}}
 \ \ \ & \text{if } \  d(Q)^2\geq 1\slash \varepsilon ;\\
 0 \ \ & \text{if } \ d(Q)^2 \leq \varepsilon;
 \end{cases}
 \end{equation*}
 \begin{equation*}
R_{j,Q,\infty}^\varepsilon
 (x,y)=\begin{cases}
 \int^{1\slash
 \varepsilon}_{d(Q)^2}\dj T_t(x,y)\frac{dt}{\sqrt{t}}
  \ \ \ & \text{if } \ \varepsilon <d(Q)^2 <1\slash \varepsilon ;\\
 0  \ & \text{if } \  d(Q)^2 \geq 1\slash \varepsilon ;\\
  \int^{1\slash
 \varepsilon}_{\varepsilon}
 \dj T_t(x,y)\frac{dt}{\sqrt{t}} \ \ & \text{if } \ d(Q)^2 \leq
 \varepsilon.
 \end{cases}
 \end{equation*}
 Clearly, $R_j^\varepsilon (x,y)= R_{j,Q, 0}^\varepsilon (x,y)+
 R_{j,Q,\infty}^\varepsilon (x,y)$ for every $Q\in \Q$ and $0<\varepsilon
 <1$. For $f\in L^1(\mathbb R^d)$ denote
 $$R_{j,Q,0}f(x) =  \lim_{\varepsilon \to 0} \int_{\mathbb R^d} R_{j,Q,0}^\varepsilon
 (x,y) f(y)\, dy, \ \ \ R_{j,Q,\infty}f(x) =\lim_{\varepsilon \to \infty}
 \int_{\mathbb R^d} R_{j,Q,0}^\varepsilon
 (x,y) f(y)\, dy,$$
 which of course exist in the sense of
 distributions.

For $Q\in \Q$ we define
 \begin{equation*}\label{eq57}
 \Q '(Q)=\{ Q'\in \Q: Q^{***} \cap
 (Q')^{***}\ne\emptyset\}, \ \ \Q ''(Q)=\{ Q''\in\Q: Q^{***}\cap (Q'')^{***}=\emptyset\}.
 \end{equation*}

 \begin{lema}\label{lem2}
Assume \ref{D} holds. Then there exists a constant $C>0$
 such that for every $Q\in \Q$ we have
 \begin{equation}\label{in21}
  \int_{\mathbb R^d}\  \sup_{0 < \varepsilon < 1}
  |R_{j,Q,\infty}^\varepsilon (x,y)|\, dx \leq C \quad \text{for }
  y\in \bigcup_{Q'\in\mathcal Q'(Q)}
  Q'^*.
 \end{equation}
 \end{lema}
 \begin{proof}
 Fix   $y\in \bigcup_{Q'\in\mathcal Q'(Q)}
  Q'^*$. Let $Q'\in \mathcal Q'(Q)$ be such that $y\in Q'^*$. Denote by $S$
  the left-hand side of (\ref{in21}). Then
 \begin{align*}
S &\leq \int_{\R^d} \int_{\min(d(Q),\, d(Q'))^2}^{d(Q')^2} \Big|
\dj T_t(x,y)\Big| \dtt dx+\int_{\R^d} \int_{d(Q')^2}^{\infty}
\Big| \dj T_t(x,y)\Big| \dtt dx\\&=S_1+S_2.\\
\end{align*}
Recall that $d(Q)\sim d(Q')$. Using \eqref{EF11}, we get
 $$S_1\leq C\int_{\min(d(Q),\, d(Q'))^2}^{d(Q')^2}t^{-1}\, dt\leq C.$$

Applying \eqref{EF11} and \ref{D},  we obtain
\begin{align*}
S_2&=\sum_{n=0}^\8 \int_{\R^d} \int_{2^n d(Q')^2}^{2^{n+1} d(Q')^2} \Big| \dj
T_t(x,y)\Big| \dtt dx\\
&\leq C \sum_{n=0}^\8 \int_{\R^d}
\int_{2^n d(Q')^2}^{2^{n+1} d(Q')^2} \int_{\R^d} \Big| \dj
T_{t-2^{n-1}d(Q')^2}(x,z)\Big| T_{2^{n-1}d(Q')^2}(z,y) \, dz \dtt dx\\
&\leq C \sum_{n=0}^\8  \int_{2^n d(Q')^2}^{2^{n+1} d(Q')^2} \int_{\R^d}
(2^{n}d(Q')^2)^{-1/2} T_{2^{n-1}d(Q')^2}(z,y) \, dz \dtt\\
&\leq C \sum_{n=0}^\8  \int_{\R^d} T_{2^{n-1}d(Q')^2}(z,y) \, dz
\leq C+C\sum_{n=1}^\8 n^{-1-\e}\leq C.
 \end{align*}
 \end{proof}
 For $0\leq \varepsilon <d(Q)^2$ let
 \begin{equation}\label{ww}
 W_{j,Q}^\varepsilon (x,y) =\int_{\varepsilon}^{d(Q)^2}
 \dj \big(T_t(x,y)-P_t(x-y)\big)\dtt.
 \end{equation}
 Set $W_{j,Q}^\varepsilon f(x)=\int W_{j,Q}^\varepsilon (x,y) f(y)\, dy$,
 $W_{j,Q} f=W_{j,Q}^{0}f$.
 \begin{lema}\label{lem1}
 Assuming \ref{K} there exists a constant $C>0$ such that
 for every $Q\in\Q$ one has
 \begin{equation*} \label{E1}
 \sup_{y\in Q^{*}}\int_{\mathbb R^d}\int_{0}^{d(Q)^2}
 \left|\dj \big(T_t(x,y)-P_t(x,y)\big)\right|\frac{dt}{\sqrt{t}}\,
 dx\leq C.
 \end{equation*}

 \end{lema}
 \begin{proof}
 The proof borrows some ideas from \cite[Lemma 2.3]{CZ}. Fix $j \in \{1,\dots , d\}$ and denote
 $$ J_Q(x,y)=\int_{0}^{d(Q)^2}
 \left|\dj \big(T_t(x,y)-P_t(x,y)\big)\right|\frac{dt}{\sqrt{t}}.$$
 The perturbation formula asserts that
 \begin{equation*}\label{perturbation}
 T_t-P_t=-\int_0^t P_{t-s}VT_s\, ds.
 \end{equation*}
 Therefore
 \begin{equation*}\begin{split}
 J_Q(x,y)&\leq \int_0^{d(Q)^2}\int_0^{t/2}\int_{\mathbb R^d}\left|
 \dj P_{t-s}(x-z)\right|V_1(z)T_s(z,y)\, dz \,
 ds\frac{dt}{\sqrt{t}}\\
 & \ \ + \int_0^{d(Q)^2}\int_{t/2}^t\int_{\mathbb R^d}\left|
 \dj P_{t-s}(x-z)\right|V_1(z)T_s(z,y)\, dz \,
 ds\frac{dt}{\sqrt{t}}\\
 & \ \ + \int_0^{d(Q)^2}\int_0^t\int_{\mathbb R^d}\left|
 \dj P_{t-s}(x-z)\right|V_2(z)T_s(z,y)\, dz \,
 ds\frac{dt}{\sqrt{t}}\\
 &=J_1'(x,y)+J_1''(x,y)+J_2(x,y),
 \end{split}\end{equation*}
 where $V_1(x)=V(x)\mathbf 1 _{Q^{***}}$, $V_2(x)=V(x)-V_1(x)$.

 To evaluate $J_1'$ observe that $$\int_{\mathbb R^d} \left|\dj P_{t-s}(x-y)\right| dx\leq
 Ct^{-1/2}\ \ \text{ for } 0<s<t/2.$$
Thus, using \ref{K}, we get
 \begin{align*}
 \int_{ Q^{**}} J_1'(x,y)\, dx
 & \leq  C  \int_0^{d(Q)^2}\int_0^{t/2} \int_{\mathbb R^d}
 t^{-1/2} V_1(z)P_s(z-y)\, dz \, ds\,
 \frac{dt}{\sqrt{t}}\cr
& \leq  C  \int_0^{d(Q)^2} t^{-1/2} \(\frac{t}{d(Q)^2}\)^{\delta} \frac{dt}{\sqrt{t}}\leq C.
 \end{align*}
Similarly,
 \begin{align*}
 \int_{ Q^{**}} J_1''(x,y)\, dx
 & \leq  C  \int_0^{d(Q)^2}\int_{t/2}^t \int_{\mathbb R^d}
 (t-s)^{-1/2} V_1(z)P_t(z-y)\, dz \, ds\,
 \frac{dt}{\sqrt{t}}\cr
&=C'  \int_0^{d(Q)^2}\int_{\mathbb R^d} V_1(z)P_t(z-y)\, dz \, dt\leq C.
 \end{align*}

  In order to estimate $J_2$ we notice that
 \begin{equation}\label{est}
 \left|\dj P_{t-s}(x-z)\right|\leq C
 d(Q)^{-d-1}e^{-c(|x-z|\slash d(Q))^2}
 \end{equation}
 for $0<s<t<d(Q)^2,  \ z\notin
 Q^{***}, \ x\in Q^{**}$.
 Lemma 3.10 of \cite{DZ5} asserts that
 $$ \sup_{y\in\mathbb R^d} \int_0^\infty \int_{\mathbb
 R^d}V(z)T_s(z,y)\, dz\, ds\leq C.$$
Hence, by  \eqref{est},  we obtain
 \begin{equation*}\begin{split}
 \int_{Q^{**}}J_2(x,y)\, dx &\leq C \, d(Q)^{-1}\int_{0}^{d(Q)^2}
 \int_0^t\int_{\mathbb R} V_2(z)T_s(z,y)\, dz\, ds\,
 \frac{dt}{\sqrt{t}}\\
 &\leq C \, d(Q)^{-1}\int_{0}^{d(Q)^2}
 \int_0^t\int_{\mathbb R} V(z)T_s(z,y)\, dz\, ds\,
 \frac{dt}{\sqrt{t}}\\
&\leq C d(Q)^{-1}\int_0^{d(Q)^2} \dtt \leq C.
 \end{split}\end{equation*}

 We now turn to estimate $J_Q(x,y)$ for
  $x\notin Q^{**}$ and $y\in Q^{*}$.
 Clearly,
 \begin{equation}\nonumber\begin{split}
 \int_{Q^{**c}} J_Q(x,y)\, dx &\leq \int_{Q^{**c}} \int_0^{d(Q)^2} \left|
 \dj T_t(x,y)\right|\frac{dt}{\sqrt{t}}\, dx\\ & +
 \int_{Q^{**c}}\int_0^{d(Q)^2} \left|
 \dj P_t(x-y)\right|\frac{dt}{\sqrt{t}}\, dx
 =\mathcal J_Q'+\mathcal J_Q''.
 \end{split}\end{equation}

  Using  (\ref{EF1}) combined with the Cauchy-Schwarz inequality
 we get
 \begin{equation}\label{Cauchy}\begin{split}
 \mathcal J_Q' & \leq \int_{0}^{d(Q)^2}\left(\int_{Q^{**c}}
 \left|\dj T_t(x,y)\right|^2e^{2|x-y|\slash
 \sqrt{t}}dx\right)^{1\slash
 2}\left(\int_{Q^{**c}}e^{-2|x-y|\slash \sqrt{t}}dx\right)^{1\slash
 2}\frac{dt}{\sqrt{t}}\\
 &\leq C \int_0^{d(Q)^2} t^{-d/4-1/2} \left(\int_{Q^{**c}}
 \left(\frac{\sqrt{t}}{|x-y|}\right)^{N}\, dx\right)^{1\slash
 2}\frac{dt}{\sqrt{t}}\leq C.
 \end{split}\end{equation}
 The estimates for $\mathcal J_{Q}''$ go in the same way. Hence
 $$ \sup_{y\in Q^{*}} \int_{Q^{**c}} J_Q(x,y)\, dx \leq C.$$
 \end{proof}
 Let
 $\{\phi_Q\}_{Q\in \Q}$ be a family of smooth functions
 that form a resolution of identity associated with
 $\{Q^*\}_{Q\in\mathcal Q}$,
 that is, $\phi_Q\in C_c^\infty (Q^*)$, $0\leq \phi_Q\leq 1$,
 $|\nabla \phi_Q(x)|\leq Cd(Q)^{-1}$, $\sum_{Q\in\mathcal Q}\phi_Q(x)=1$ a.e.

 The following corollary  follows easily from Lemma \ref{lem1}.
 \begin{coro}
 For $f\in L^1(\mathbb R^d)$ we have
 \begin{equation*}\label{eq71}\lim_{\varepsilon \to 0} \| W_{j,Q}^{\varepsilon} (\phi_Q
 f) - W_{j,Q}(\phi_Q f)\|_{L^1(\mathbb R^d)}=0 \ \ \text{and} \ \
 \| W_{j,Q} (\phi_Q f)\|_{L^1(\mathbb R^d)}\leq C \| \phi_Q
 f\|_{L^1(\mathbb R^d)}
 \end{equation*}
 with $C$ independent of $Q$ and $f$.
 \end{coro}
 \begin{lema}\label{lem3}
 There exists a constant $C>0$ such that for every $Q\in\Q$
 and every $f\in L^1(\mathbb R^d    )$ such that $\text{supp}\,
 f\subset \tilde Q =\bigcup_{Q'\in \mathcal Q'(Q)} Q'^{*}$ we have
 \begin{equation}\label{eq51}
 \| R_j(\phi_Q f)-\phi_Q R_jf\|_{L^1(\mathbb R^d)}\leq C\|
 f\|_{L^1(\tilde Q)}.
 \end{equation}
 \end{lema}
 \begin{proof}
Note that
 \begin{equation*}
 R_j(\phi_Q f)(x)-\phi_Q (x) R_jf (x) =\lim_{\varepsilon \to 0}
 \int_{\varepsilon}^{1\slash \varepsilon} \int \Big(\dj T_t(x,y)\Big)\left(\phi_Q(y)-\phi_Q(x)\right)f(y)\, dy
 \frac{dt}{\sqrt{t}}.
 \end{equation*}
 From (\ref{EF1}) we conclude
 \begin{equation}\begin{split}\label{eq53}
 \int_{\mathbb R^d}\int_{0}^{d(Q)^2} &
\left|\Big(\dj T_t(x,y)\Big) \left(\phi_Q(y)  -\phi_Q(x)\right) \right| \frac{dt}{\sqrt{t}}\,
 dx\\
 & \leq \frac{C}{d(Q)}\int_{\mathbb R^d}\int_{0}^{d(Q)^2}
 \left|\dj T_t(x,y)\right|\frac{|x-y|}{\sqrt{t}}\, dt\,
 dx\\
 &\leq \frac{C}{d(Q)}\int_{\mathbb R}\int_{0}^{d(Q)^2}
 \left|\dj T_t(x,y)\right|e^{|x-y|\slash \sqrt{t}}\, dt\,
 dx
   \leq C.
 \end{split}\end{equation}
 Now (\ref{eq51}) follows from (\ref{in21}) and (\ref{eq53}).
 \end{proof}

 The following lemma is motivated by \cite[Lemma 3.8]{DZ5}.

 \begin{lema}\label{lem4}
 There exists a constant $C>0$ such that
 \begin{equation}\label{eq58}
 \sum_{Q\in\Q} \Big\| \mathbf 1_{Q^{***}}R_j\Big(\sum_{Q''\in
 \Q ''(Q)} \phi_{Q''}f\Big)\Big\|_{L^1(\mathbb R^d)}\leq C \| f\|_{L^1(\mathbb R^d)}.
 \end{equation}
 \end{lema}
 \begin{proof} Let $S$ denote the left-hand side of (\ref{eq58}).
 Applying (\ref{cov}), we have
 \begin{equation}\begin{split}\label{eq90}
 S &
 \leq \sum_{Q\in\Q}\ \sum_{Q''\in\Q ''(Q)} \Big\| \mathbf
 1_{Q^{***}}
 R_j(\phi_{Q''}f)\Big\|_{L^1(\mathbb R^d)}\\
 & =\sum_{Q''\in\Q}\ \sum_{Q\in\Q ''(Q'')} \Big\| \mathbf
 1_{Q^{***}}
 R_j(\phi_{Q''}f)\Big\|_{L^1(\mathbb R^d)}\\
 &\leq C\sum_{Q''\in\Q}\  \Big\|
 R_j(\phi_{Q''}f)\Big\|_{L^1((Q'')^{**c})}\\
 &\leq C\sum_{Q''\in\Q}\  \Big\|
 R_{j,Q'',0}(\phi_{Q''}f)\Big\|_{L^1((Q'')^{**c})}+C\sum_{Q''\in\Q}\  \Big\|
 R_{j,Q'',\infty} (\phi_{Q''}f)\Big\|_{L^1((Q'')^{**c})}.
 \end{split}\end{equation}
 Using (\ref{in21}) and (\ref{cov}),  we get
 \begin{equation}\label{eq99} \sum_{Q''\in\Q}\  \Big\|
 R_{j,Q'',\infty} (\phi_{Q''}f)\Big\|_{L^1((Q'')^{**c})}\leq C \sum_{Q''\in\Q}\| \phi_{Q''}f\|_{L^1(\mathbb R^d)}\leq C'\| f\|_{L^1(\mathbb R^d)}.
 \end{equation}
Identically as in \eqref{Cauchy} for $y\in (Q'')^{*}$ we have
 $$
 \int_{(Q'')^{**c}} \int_{0}^{d(Q'')^2}\left| \dj T_t(x,y) \right|\frac{dt}{\sqrt{t}} \, dx\leq
 C,
 $$
 which implies
\begin{equation}\label{eq91} \sum_{Q''\in\Q}\  \Big\|
 R_{j,Q'',0} (\phi_{Q''}f)\Big\|_{L^1((Q'')^{**c})}\leq
 C\sum_{Q''\in\Q} \| \phi_{Q''}f\|_{L^1(\mathbb R^d)}\leq C\|
 f\|_{L^1(\mathbb R^d)}.
 \end{equation}
 The lemma is a consequence of (\ref{eq90})-(\ref{eq91}).
  \end{proof}

 \section{Proof of Theorem  \ref{main} }

 In order to prove the second inequality of (\ref{eq18}) it suffices
 by (\ref{estimate_R}) and (\ref{atoms})  to verify that  there exists a constat $C>0$
 such that
 \begin{equation}\label{eq31}
 \| R_j a\|_{L^1(\mathbb R^d)} \leq C
 \end{equation}
 for every $H^1_{\Q}$-atom $a$ and $j=1, \dots ,d$. Assume that $a$ is an $H^1_{\Q}$-atom
  supported by a cube $Q^*$, $Q\in\Q$. Then
 \begin{equation*}\begin{split}
 R_j a(x)& = \lim_{\varepsilon \to 0} \left( R_{j,Q,0}^\varepsilon a (x)
 + R_{j,Q,\infty}^\varepsilon a (x)\right)\\
 &= \lim_{\varepsilon \to 0} \left(W_{j,Q}^\varepsilon a(x)
  +H_{j,Q}^{\varepsilon}a(x)+ R_{j,Q,\infty}^\varepsilon a
 (x)\right),\\
 \end{split}\end{equation*}
  where $H_{j,Q}^\varepsilon a(x)= \int_{\varepsilon}^{d(Q)^2}
 \dj (a*P_t)(x)\dtt$. Similarly to
 (\ref{estimate_R}),
 the limit
 $$H_{j,Q} a(x)=\lim_{\varepsilon\to 0} H_{j,Q}^\varepsilon a(x)$$
 exists in the sense of distributions.
 Moreover, by the boundedness of the local
 Riesz transforms on the local Hardy spaces (see \cite{Gold}), we
 have $\| H_{j,Q}a \|_{L^1(\mathbb R^d)} \leq C$ with $C$ independent of $a$.
 Using Lemmas \ref{lem1} and \ref{lem2},  we
 obtain (\ref{eq31}).

 We now turn to prove the first inequality of (\ref{eq18}).
 To this end, by the local Riesz transform characterization of the local Hardy spaces
 (see \cite[Section 2]{Gold}),  it suffices to show that
\begin{equation}\label{eq41}
 \sum_{Q\in\Q} \| H_{j,Q}(\phi_Q f )\|_{L^1(Q^{**})}\leq
 C\left(\| f\|_{L^1(\mathbb R^d)} + \| R_j f\|_{L^1(\mathbb R^d)}\right), \ \ j=1,...,d.
\end{equation}

 Clearly,
 \begin{equation*}\label{eq42}
 H_{j,Q}(\phi_Q f) = - W_{j,Q}(\phi_Q f)+R_{j,Q,0}(\phi_Q f).
 \end{equation*}
 Lemma \ref{lem1} together with (\ref{cov}) implies
 \begin{equation}\label{eq60}
 \sum_{Q\in\Q} \| W_{j,Q} (\phi_Q f)\|_{L^1(\mathbb R^d)}\leq
 C\sum_{Q\in\Q} \| \phi_Q f\|_{L^1(\mathbb R^d)} \leq C \|
 f\|_{L^1(\R^d)}.
 \end{equation}
 Note that
 \begin{equation}\begin{split}\label{eq43}
 R_{j,Q,0}(\phi_Q f)
 &=-R_{j,Q,\infty} (\phi_Q f)
 + \Big[R_j \Big(\phi_Q  \sum_{Q'\in\Q '(Q)} (\phi_{Q'} f)\Big)-\phi_{Q}
 R_j\Big(\sum_{Q'\in\Q '(Q)}
 (\phi_{Q'}f)\Big)\Big]\\
 &\ \ -\phi_Q R_j\Big(\sum_{Q''\in\Q ''(Q)} (\phi_{Q''}
 f)\Big) + \phi_Q R_j f.
 \end{split}\end{equation}
  Lemmas \ref{lem2}, \ref{lem3}, and \ref{lem4} combined with
  (\ref{eq43}) imply
 \begin{equation}\label{eq61}\begin{split}
 \sum_{Q\in\Q}\| R_{j,Q,0} (\phi_Q f)\|_{L^1(Q^{**})} & \leq
 C\Big(
 \sum_{Q\in\Q} \| \phi_Q f\|_{L^1(\R^d)}+\sum_{Q\in\Q}\sum_{Q'\in\Q '(Q)} \| \phi_{Q'} f\|_{L^1(\R^d)}\\
 &\ \ \ \  + \|
 f\|_{L^1(\R^d)}
   +\sum_{Q\in\Q} \| \phi_Q R_j f\|_{L^1(\R^d)}\Big)\\
   & \leq C\left( \| f\|_{L^1(\R^d)} + \| R_j f\|_{L^1(\R^d)}\right).
  \end{split}\end{equation}
 Now (\ref{eq41}) follows from (\ref{eq60}) and (\ref{eq61}).

\section{Proof of Lemma \ref{lem2.1}}
 The proof is based on estimates of the semigroup $T_t$ acting on
 weighted $L^2$ spaces. This technique was utilize e.g. in
 \cite{DH}, \cite{hebisch}, \cite{DZ56}.

Fix $y_0\in\mathbb R^d$ and $\alpha >0$. The semigroup
$\{T_t\}_{t>0}$ acting on $L^2(e^{\alpha |x-y_0|}dx)$ has the
unique extension to a holomorphic semigroup $T_\zeta$, $\zeta\in
\{\zeta\in \mathbb C: |\text{Arg}\, \zeta|<\pi\slash 4\}$ such
that
\begin{equation}\label{holo1}
 \| T_\zeta \|_{L^2(e^{\alpha |x-y_0|}dx)\to L^2(e^{\alpha |x-y_0|}dx)
 }\leq Ce^{c'\alpha^2\Re \zeta}
\end{equation}
with $C$ and $c'$ independent of $V$ and $y_0$ (see, e.g.,
\cite[Section 6] {DZ56}). Let $-\mathcal L_\alpha $ denote the
infinitesimal generator of
 $\{T_t\}_{t>0}$ considered on $L^2(e^{\alpha |x-y_0|}dx)$. The
 quadratic form $\mathbf Q=\mathbf Q_{\alpha \, ,y_0}$
 associated with $\mathcal L_\alpha$ is given by
 \begin{equation}\begin{split}\label{holo2}
\mathbf Q(f,g)& =\sum_{j=1}^d \int_{\mathbb R^d}
\frac{\partial}{\partial x_j}f(x)\frac{\partial}{\partial
 x_j}\overline{g(x)} e^{\alpha |x-y_0|}\, dx + \int_{\mathbb R^d}
 V(x)f(x)\overline{g(x)}e^{\alpha |x-y_0|}\, dx\\
 &\ \ +\sum_{j=1}^d \int_{\mathbb R^d} f(x)
 \overline{g(x)}\frac{\partial}{\partial x_j} e^{\alpha|x-y_0|}\,
 dx,
 \end{split}\end{equation}
$$D(\mathbf Q) =\{ f\in L^2(e^{\alpha |x-y_0|}dx): \,
V(x)^{1\slash 2} f(x), \ \frac{\partial}{\partial x_j} f(x)\in
L^2(e^{\alpha |x-y_0|}dx), \ j=1,...,d\}.$$
 Note that
 $$\left|\frac{\partial}{\partial x_j}e^{\alpha |x-y_0|}\right|\leq C
 \alpha e^{\alpha |x-y_0|} \ \ \  \text{for }  x\ne y_0.$$
  Clearly,
 \begin{equation}\label{holo3}
 |\mathbf Q(f,g)|\leq C_\alpha \| f\|_{\mathbf Q}\| g\|_{\mathbf Q}
 \end{equation}
 with $C_\alpha $ independent of $y_0$ and $V$, where
 $$ \| f\|_{\mathbf Q}^2= \int_{\mathbf R^d}
 \left( \sum_{j=1}^d \left|\frac{\partial}{\partial x_j} f(x)\right|^2
 +V(x)|f(x)|^2+|f(x)|^2\right)e^{\alpha |x-y_0|}dx.$$
 Moreover, there exists a constant $C>0$ independent of $V$ and
 $y_0$ such that
 \begin{equation}\label{holo4}
 \| f\|_{\mathbf Q}^2\leq C\mathbf Q(f,f).
 \end{equation}
 The holomorphy of the semigroup $T_t$ combined with (\ref{holo1})
 imply
 \begin{equation}\label{holo5}
 \| \mathcal L_\alpha T_t g\|_{L^2(e^{\alpha |x-y|}dx)}\leq
 C't^{-1} e^{c''t\alpha^2 } \|g\|_{L^2(e^{\alpha |x-y_0|}dx)}
 \end{equation}
 with constants $C'$ and $c''$ independent of $V$ and $y_0$.
 Setting $g(x)=T_{1\slash 2}(x,y_0)$, $f(x)=T_{1\slash 2}
 g(x)=T_1(x,y_0)$ and using  (\ref{holo4}), (\ref{holo5}), (\ref{holo1}), and (\ref{kac}), we get
 \begin{equation}\label{holo6}
 \begin{split}
 \left\| \frac{\partial}{\partial x_j}T_1(x,y_0)\right
 \|^2_{L^2(e^{\alpha |x-y_0|}dx)} & \leq \| f\|_{\mathbf Q}^2\\
 & \leq C \mathbf Q(f,f) \\
 &\leq C \| \mathcal L_\alpha f\|_{L^2(e^{\alpha |x-y_0|}dx)} \|
 f\|_{L^2(e^{\alpha |x-y_0|}dx)}\\
 &\leq C''\| g\|_{L^2(e^{\alpha |x-y_0|}dx)}^2\leq C'''
 \end{split}\end{equation}
 with $C'''$ independent of $y_0$ and $V$. Since
 $T_t(x,y)=t^{-d\slash 2}\tilde T_1(x\slash \sqrt{t},y\slash
 \sqrt{t})$, where $\{\tilde T_s\}_{s>0}$ is the semigroup
 generated by $\Delta - tV(\sqrt{t}x)$, we get (\ref{EF1}) from
 (\ref{holo6}), because $C'''$ is independent of $V$ and $y_0$.
 Now (\ref{EF11}) follows from (\ref{EF1}) and the
 Cauchy-Schwarz inequality.

\end{document}